\documentclass[12pt]{article}
\usepackage{amssymb, amsmath}
\newcommand{\qed}{\hfill \rule{2.5mm}{2.5mm}}

\newcommand{\N}{{\mathbb N}}

\begin{document}
\newtheorem{thm}{Theorem}[section]
\newtheorem{defs}[thm]{Definition}
\newtheorem{lem}[thm]{Lemma}
\newtheorem{note}[thm]{Note}
\newtheorem{cor}[thm]{Corollary}
\newtheorem{prop}[thm]{Proposition}
\renewcommand{\theequation}{\arabic{section}.\arabic{equation}}
\newcommand{\newsection}[1]{\setcounter{equation}{0} \section{#1}}
\title{Strong completeness  of a class of $L^2$-type Riesz spaces
      \footnote{{\bf Keywords:} Strong completeness; Riesz spaces; conditional expectation operators; Riesz representation; strong dual.\
      {\em Mathematics subject classification (2010):} 46B40; 60F15; 60F25.}}
\author{Wen-Chi Kuo\footnote{Supported in part by  National Research Foundation of South Africa grant number CSUR160503163733.},
School of Mathematics\\
University of the Witwatersrand\\
Private Bag 3, P O WITS 2050, South Africa\\ \\
Anke Kalauch\\
TU-Dresden\\
\\ Bruce A. Watson \footnote{Supported in part by the Centre for Applicable Analysis and
Number Theory
.} \\
School of Mathematics\\
University of the Witwatersrand\\
Private Bag 3, P O WITS 2050, South Africa }
\maketitle
\abstract{\noindent
Strong convergence and convergence in probability were generalized to the setting of a Riesz space with conditional expectation operator, $T$, in [{{\sc Y. Azouzi, W.-C. Kuo, K. Ramdane, B. A. Watson}, {Convergence in Riesz spaces with conditional expectation operators}, {\em Positivity}, {\bf 19} {(2015), 647-657}}] as $T$-strong convergence and convergence in $T$-conditional probability, respectively. Generalized $L^{p}$ spaces for the cases of $p=1,2,\infty$,  were discussed in the setting of Riesz spaces as ${L}^{p}(T)$ spaces in [{{\sc C. C. A. Labuschagne, B. A. Watson}, {Discrete stochastic integration in Riesz spaces}, {\em Positivity}, {\bf 14} {(2010), 859-875}}].
An $R(T)$ valued norm, for the cases of $p=1,\infty,$ was introduced on these spaces in 
[{{\sc W. Kuo, M. Rogans, B.A. Watson}, {Mixing processes in Riesz spaces}, {\em Journal of Mathematical Analysis and Application}, {\bf 456} {(2017), 992-1004}}]
where it was also shown that $R(T)$ is a universally complete $f$-algebra and that these spaces are $R(T)$-modules.
In [{{\sc Y. Azouzi, M. Trabelsi}, {$L^p$-spaces with respect to conditional expectation on Riesz spaces}, {\em Journal of Mathematical Analysis and Application}, {\bf 447} {(2017), 798-816}}] functional calculus was used to consider  ${L}^{p}(T)$  for $p\in (1,\infty)$.
The strong sequential completeness of the space ${L}^{1}(T)$, the natural domain of the conditional expectation operator $T$,
and the strong completeness of  ${L}^{\infty}(T)$ was established in 
[{{\sc W.-C. Kuo, D. Rodda, B. A. Watson}, {Sequential strong completeness of the natural domain of Riesz space conditional expectation operators}, {\em Proc. AMS}, {\bf 147} {(2019), 1597–1603}}].
In the current work the $T$-strong completeness of ${L}^{2}(T)$ is established along with a Riesz-Fischer type theorem where the duality is with
respect to the $T$-strong dual. It is also shown that the conditional expectation operator $T$ is a weak order unit for the $T$-strong dual.}
\parindent=0in
\parskip=.2in

\section{Introduction}

Strong convergence was generalized to Dedekind complete Riesz spaces with a conditional expectation operator in \cite{AKRW} as $T$-strong convergence.
Generalized $L^{p}$ spaces for $p=1,2,\infty$ were discussed in the setting of Riesz spaces as ${L}^{p}(T)$ spaces in \cite{LW}. An $R(T)$ valued norm, for the cases of $p=1,\infty,$ was introduced on the ${L}^{p}(T)$  spaces in \cite{KRW} where it was also shown that $R(T)$ is a universally complete $f$-algebra and that these spaces are $R(T)$-modules.
More recently, in \cite{AT}, the ${L}^{p}(T)$, for $p\in (1,\infty)$, spaces were considered. We also refer the reader to \cite{JHVDW} for an interesting study of sequential order convergence in 
vector lattices using convergence structures and filters.
In \cite{KRodW} the $T$-strong sequential completeness of the natural domain,  ${L}^{1}(T)$, of the Riesz space conditional expectation operator $T$
was established, i.e. that each $T$-strong Cauchy sequence in ${L}^{1}(T)$ converges $T$-strongly in ${L}^{1}(T)$.

\begin{defs}
	We say that a net  $(f_\alpha)$ in ${\mathcal{L}}^p(T)$, where $p\in [0,\infty]$, 
	is a strong Cauchy net if 
	$$v_\alpha:=\sup_{\beta,\gamma\ge \alpha}\|f_\beta-f_\gamma\|_{T,p}$$
	is eventually defined in $R(T)$ and has order limit zero.
\end{defs}

The term $T$-strong here means  with respect to the $R(T)$ valued norm induced by the conditional expectation operator $T$ in the given space.  It was also shown that  ${L}^{\infty}(T)$ is $T$-strongly complete i.e. that every $T$-strong Cauchy net in ${L}^{\infty}(T)$ is $T$-strongly convergent. 

In the current work the $T$-strong completeness of ${L}^{2}(T)$ is established, i.e. that each net in ${L}^{2}(T)$ which is Cauchy with respect to the $R(T)$-valued norm, $f\mapsto (T|f|^2)^{1/2}:=\|f\|_{T,2}$,  is convergent in ${L}^{2}(T)$ with respect to this norm.
This is proved via the a Riesz-Fischer type theorem where the duality is with
respect to the $T$-strong dual of ${L}^{2}(T)$. It is also shown that the conditional expectation operator $T$ is a weak order unit for the $T$-strong dual of ${L}^{2}(T)$.

The issue of completeness of ${L}^{2}(T)$ is important in the
theory of stochastic integrals in Riesz spaces, since these integrals are defined
to be limits of Cauchy nets in ${L}^{2}(T)$,
see for example  \cite{GL}.
The results also impact on the study of martingales in Riesz spaces, see \cite{Stoica, Troitsky}.

\section{Preliminaries}

Throughout this work $E$ will denote a Dedekind complete Riesz space with weak order unit and $T$ will denote a strictly positive conditional expectation operator on $E$. 
By $T$ being a conditional expectation
operator on $E$ we mean that $T$ is a linear positive order continuous projection on $E$ which maps weak order units to weak order units and has range $R(T)$ closed with respect to order limits in $E$. 
This gives that there is at least one weak order unit, say $e$, with $Te=e$, and that 
$R(T)$ is Dedekind complete with considered as a subspace of $E$. 
By $T$ being strictly positive we mean that if
$f\in E_+$, the positive cone of $E$, and $f\ne 0$ then $Tf\in E_+$ and $Tf\ne 0$.

A strictly positive conditional expectation operator, $T$, on a Dedekind complete Riesz space with weak order unit, can be extended to a strictly positive conditional expectation operator, also denoted $T$,  on its natural domain, denoted $L^1(T):=\mbox{dom}(T)-\mbox{dom}(T)$.
We say that $E$ is $T$-universally complete if $E=L^1(T)$. 
From the definition of $\mbox{dom}(T)$, see \cite{KLW-exp}, $E$ is $T$-universally complete if and only
 if for each upwards directed net $(f_{\alpha})_{\alpha \in \Lambda}$ in $E^+$ such that $(Tf_{\alpha})_{\alpha \in \Lambda}$ is order bounded in $E_{u}$, we have that $(f_{\alpha})_{\alpha \in \Lambda}$ is order convergent in $E$. 
 Here $E_u$ denotes the universal completion of $E$, see \cite[page 323]{L-Z}. 
 $E_u$ has an $f$-algebra structure which can be chosen so that $e$ is the multiplicative identity.
For $T$ acting on $E=L^1(T)$,  $R(T)$ is a universally complete $f$-algebra and $L^1(T)$ is an $R(T)$-module.
 From \cite[Theorem 5.3]{KLW-exp}, $T$ is an averaging operator, i.e. if $f\in R(T)$ 
and $g\in E$ then $T(fg)=fT(g)$.
 This prompts the definition of an $R(T)$ (vector valued) norm $\|\cdot\|_{T,1}:=T|\cdot|$ on $L^1(T)$. 
The homogeneity is with respect to multiplication by elements of $R(T)^+$.

The definition of the Riesz space
 $L^{2}(T):=\{f\in L^1(T)\,|\, f^2\in L^1(T)\}$ was given in  
 \cite{LW}.
 By the averaging property,  $L^{2}(T)$ is an $R(T)$-modules and the map
 $$f\mapsto\|f\|_{T,2}:=(T(f^2))^{1/2}, \quad f\in L^{2}(T)$$
 is an $R(T)$-valued norm on $L^{2}(T)$. Aspects for this development for $L^{p}(T)$ with general  $1<p<\infty$ can be found in \cite{AT, grobler-1}. Here the multiplication is as defined in the $f$-algebra $E_u$.
Proofs of various H\"older type inequalties in Riesz spaces with conditional expectation operators can be found in \cite{KRW} and \cite{AT}. In particular,
\begin{equation}\label{holder}
T|fg|\le \|f\|_{T,2}\|g\|_{T,2}, \, \mbox{ for all } f,g\in {{L}}^2(T).
\end{equation}

If $F$ is a universally complete Riesz space with weak order unit, say $e$, then $F$ is an $f$-algebra and $e$ can be taken as the algebraic unit, see \cite[Theorem 3.6]{V-E}.
 We recall from \cite[Appendix]{KKW-1} the following material on partial inverses in Riesz spaces.

\begin{defs}
Let $F$ be a  universally complete Riesz space with weak order unit, say $e$ and take $e$ as the algebraic unit of  the associated $f$-algebra structure. 
We say that $g\in F$ has a partial inverse if there exists $h\in F$ such that $gh=hg=P_{|g|}e$ where $P_{|g|}$ denotes the band projection onto the band generated by $|g|$. 
We refer to $h$ as the canonical partial inverse of $g$ if in addition to being a partial inverse to $g$, we have that $(I-P_{|g|})h=0$, i.e.  $h\in {\cal B}_{|g|}$,  where ${\cal B}_{|g|}$ is the band generated by $|g|$.
\end{defs}

The following result gives existence, uniqueness and positivity results concerning partial inverses and canonical partial inverses. 
We denote by ${\cal B}_f$ the band generated by $f$ and by $P_f$ the band projection onto ${\cal B}_f$.

 \begin{thm}
Let $F$ be a  universally complete Riesz space with weak order unit, say $e$, which also take as the algebraic unit of  the associated $f$-algebra structure.  Each $g\in F$ has a partial inverse $h\in F$. 
The canonical partial inverse of $g$ is unique and in this case $g$ is also the canonical partial inverse of $h$. If $g\in F^+$ then so is its canonical partial inverse.
  \end{thm}
\section{Dual spaces}

Let $E=L^{2}(T)$. We say that a map 
${\frak f}:E\to R(T)$
is a $T$-linear functional on $E$ if it is additive, $R(T)$-homogeneous and order bounded.
We denote the space of $T$-linear functionals on $E$ by $E^*$ and call it the $T$-dual of $E$. 
We note that $E^*\subset {\cal L}_b(E, R(T))$, since $R(T)$-homogeneity implies real linearity. 
Further as $R(T)$ is Dedekind complete, so is ${\cal L}_b(E, R(T))$, see \cite[page 12]{AB}, and
${\cal L}_b(E, R(T))={\cal L}_r(E, R(T))$, where ${\cal L}_r(E, R(T))$ denotes the regular operators, see \cite[page 10]{AB}. 
Following the notation of \cite{AB}, we denote the order continuous elements of ${\cal L}_b(E, R(T))$ by ${\cal L}_n(E, R(T))$ and by \cite[page 44]{AB},  ${\cal L}_n(E, R(T))$ is a band in  ${\cal L}_b(E, R(T))$, and since  ${\cal L}_b(E, R(T))$ is Dedekind complete, so is  ${\cal L}_n(E, R(T))$.

If ${\frak f}\in E^*$ and 
there is $k\in R(T)^+$ such that 
\begin{equation}
|{\frak f}(g)|\le k\|g\|_{T,2},\quad\mbox{for all } g\in E,\label{domination}
\end{equation}
we say that ${\frak f}$ is $T$-strongly bounded. 
We denote the space of $T$-strongly bounded $T$-linear functionals on $E$ by 
\begin{eqnarray*}
\hat{E}:=\{{\frak f}\in E^*\,|\, {\frak f} \mbox{ $T$-strongly bounded }\}
\end{eqnarray*}
and refer to it as the $T$-strong dual of $E$. 
Further 
$$\|{\frak f}\|:=\inf\{k\in R(T)^+\,|\,|{\frak f}(g)|\le k\|g\|_{T,2}\quad\mbox{for all } g\in E\}$$
defines and $R(T)$-valued norm on $\hat{E}$ with 
\begin{eqnarray}\label{norm-bound}
|{\frak f}(g)|\le \|{\frak f}\|\,\|g\|_{T,2}
\end{eqnarray}
 for all $g\in L^2(T)$.

We note here that as the map $g\mapsto \|g\|_{T,2}$ is order continuous the domination in (\ref{domination}) gives that each ${\frak f}\in \hat{E}$ is order continuous. Thus $\hat{E}\subset E^*\cap {\cal L}_n(E, R(T))$. 

From \cite{KKW-1} we have the following Riesz-Frechet representation theorem.

\begin{thm}[Riesz-Frechet representation theorem in Riesz space]\label{thm-final}
 The map $\Psi$ defined by $\Psi(f)(g):=T_f(g)=T(fg)$ for $f,g\in {{L}}^2(T)$ is a bijection between $E={{L}}^2(T)$ and, its $R(T)$-homogeneous strong dual,
 $\hat{E}$. This map is additive, $R(T)$-homogeneous 
 and $R(T)$-valued norm preserving in the sense that $\|T_f\|=\|f\|_{T,2}$ for all $y\in {{L}}^2(T)$.
\end{thm}

A direct application of Theorem \ref{thm-final} gives that $T\in \hat{E}$ since $\Psi(e)=T$ and
$L^2(T)$ is an $R(T)$ module in which $ef=f$ for all $f\in L^2(T)$.

\begin{lem}[Riesz-Kantorovich]\label{lem-RK}
The space $E^*$ is a Riesz space with respect to the partial ordering ${\frak f}\le {\frak g}$ if and only if ${\frak f}(x)\le {\frak g}(x)$ for all $x\in E_+$. This partial ordering is equivalent to defining the lattice operations by
$$({\frak f}\vee {\frak g})(x):=\sup\{{\frak f}(y)+{\frak g}(z)\,|\,y,z\in E_+, y+z=x\}$$ 
and
$$({\frak f}\wedge {\frak g})(x):=\inf\{{\frak f}(y)+{\frak g}(z)\,|\,y,z\in E_+, y+z=x\}$$ 
for all $x\in E_+$ and extending these operators to $E$ by the Kantorovich Theorem, \cite[page 7]{AB}.
Here ${\frak f}_\alpha\downarrow {\frak 0}$ in $E^*$ if and only if ${\frak f}_\alpha(x)\downarrow 0$ in $E$ for each $x\in E_+$.
$E^*$ is Dedekind complete and an $R(T)$-module.
\end{lem}

\begin{proof}
 If ${\frak f},{\frak g}\in E^*$ then 
 ${\frak f},{\frak g}\in {\cal L}_b(E, R(T))$ and hence (as ${\cal L}_b(E, R(T))$ is a Riesz space) ${\frak f}+{\frak g}\in {\cal L}_b(E, R(T))$.
 Further, for $\alpha\in R(T)$ and $x\in E$, since ${\frak f}$ and ${\frak g}$ are $R(T)$-homogeneous, we have that 
 $$({\frak f}+{\frak g})(\alpha x):={\frak f}(\alpha x)+{\frak g}(\alpha x)=\alpha({\frak f}(x)+{\frak g}(x))=:\alpha ({\frak f}+{\frak g})(x).$$
Thus ${\frak f}+{\frak g}\in E^*$.

To show that $E^*$ is a sublattice of ${\cal L}_b(E, R(T))$ it remains to be shown that ${\frak f}\vee {\frak g}\in E^*$ for each ${\frak f},{\frak g}\in E^*$, with ${\frak f}\vee {\frak g}$ as defined in the lemma statement. 
This also ensures that ${\frak f}\vee_{{\cal L}_b(E, R(T))}{\frak g}={\frak f}\vee_{E^*}{\frak g}$. 
As $E^*\subset {\cal L}_b(E, R(T))$ to show that
${\frak f}\vee {\frak g}\in E^*$ is the same is 
  ${\frak f}\vee_{{\cal L}_b(E, R(T))}{\frak g}=:{\frak f}\vee {\frak g}$ 
  it suffices to show that ${\frak f}\vee{\frak g}$ is in $E^*$. 
  Since ${\frak f}\vee{\frak g}\in {{\cal L}_b(E, R(T))}$ it remains only to show that it is $R(T)$ homogeneous. 
  
For $\alpha\in R(T)_+$ and $x\in E_+$ we have
$$({\frak f}\vee{\frak g})(\alpha x):=\sup\{{\frak f}(y)+{\frak g}(z)\,|\,y,z\in E_+, y+z=\alpha x\}.$$ 
If $y,z\in E_+$ and $y+z=\alpha x$, then $y,z \in B_\alpha$ where $B_\alpha$ is the band in $E$ generated by $\alpha$. 
 
Now as $R(T)$ is universally complete, $\alpha$ has a partial inverse $\beta\in R(T)_+\cap B_\alpha$ such that $$\alpha\cdot \beta =\beta\cdot \alpha= P_\alpha e$$
where $P_\alpha$ denotes the band projection onto the band $B_\alpha$, generated by $\alpha$.
Thus $\beta y+\beta z=P_\alpha x$ and $\alpha\beta y=y$ and similarly for $z$.  
Hence
\begin{eqnarray*}({\frak f}\vee{\frak g})(\alpha x)&=&\sup\{\alpha({\frak f}(\beta y)+{\frak g}(\beta z)\,|\,y,z\in E_+,\beta y+\beta z=P_\alpha x\}\\
&\le&  
\sup\{\alpha({\frak f}(y')+{\frak g}(z')\,|\,y',z'\in E_+,y'+z'=P_\alpha x\}\\
&\le&  
\sup\{\alpha({\frak f}(y')+{\frak g}(z')\,|\,y',z'\in E_+,y'+z'=x\}\\
&=&\alpha ({\frak f}\vee{\frak g})(x)
\end{eqnarray*}
and
\begin{eqnarray*}
\alpha\cdot ({\frak f}\vee{\frak g})(x)&=&\sup\{\alpha\cdot ({\frak f}(y)+{\frak g}(z))\,|\,y,z\in E_+, y+z=x\}\\
&=&\sup\{({\frak f}(\alpha y)+{\frak g}(\alpha z))\,|\,y,z\in E_+, y+z=x\}\\
&\le &\sup\{({\frak f}(\alpha y)+{\frak g}(\alpha z))\,|\,y,z\in E_+, \alpha y+\alpha z=\alpha x\}\\
&\le &\sup\{({\frak f}(y')+{\frak g}(z'))\,|\,y',z'\in E_+, y'+z'=\alpha x\}\\
&=&({\frak f}\vee{\frak g})(\alpha x)
\end{eqnarray*}
Thus $({\frak f}\vee{\frak g})(\alpha x)=\alpha ({\frak f}\vee{\frak g})(x)$ for $\alpha \in R(T)_+$ and $x\in E_+$. 
After proving the analogous result for ${\frak f}\wedge{\frak g}$ and as we know $({\frak f}\vee{\frak g})(-x)=-({\frak f}\wedge{\frak g})(x)$ the homogeneity follows for all $x\in E$ and $\alpha \in R(T)$.

Finally we show that $E^*$ is Dedekind complete. 
As ${\cal L}_b(E, R(T))$ is Dedekind complete, it suffices to show that $E^*$ is order closed in ${\cal L}_b(E, R(T))$.
 If $({\frak f}_\gamma)$ is a net in $E^*$ with order limit ${\frak f}$ in ${\cal L}_b(E, R(T))$ then for each $\alpha \in R(T)$ we have that ${\frak f}_\gamma(\alpha x)=\alpha{\frak f}_\gamma$. 
 Further by order continuity of  multiplication by elements of $R(T)$ it follows that the net $(\alpha {\frak f}_\gamma)$ has order limit $\alpha {\frak f}$. 
 However, for each $x\in E$, $({\frak f}_\gamma(\alpha x))$ has order limit ${\frak f}(\alpha x)$ and $(\alpha{\frak f}_\gamma(x))$ has order limit $\alpha{\frak f}(x)$ in $R(T)\subset E$. 
 Thus ${\frak f}(\alpha x)=\alpha{\frak f}(x)$ giving ${\frak f}\in E^*$.
\qed
\end{proof}

\begin{lem}
The space $\hat{E}$ is  an $R(T)$-module and a Dedekind complete Riesz subspace of $E^*$.
\end{lem}

\begin{proof}
From its definition it is clear that $\hat{E}$ is a vector subspace of $E^*$ and is an $R(T)$-module.
It remains to prove that $\hat{E}$ is a sublattice of $E^*$ and that it is Dedekind complete.

To show that $\hat{E}$ is a sublattice of $E^*$ we observe that if ${\frak f},{\frak g}\in \hat{E}$ then there exist $k_{\frak f},k_{\frak g}\in R(T)_+$ so that 
$|{\frak f}(x)|\le k_{\frak f}\|x\|_{T,2}$ and $|{\frak g}(x)|\le k_{\frak g}\|x\|_{T,2}$ for all $x\in E$. 
From Lemma \ref{lem-RK},
$$|({\frak f}\vee {\frak g})(x)|\le\sup\{|{\frak f}(y)+{\frak g}(z)|\,|\,y,z\in E_+, y+z=x\}\le (k_{\frak f}+k_{\frak g})\|x\|_{T,2}$$ 
for all $x\in E$, giving ${\frak f}\vee{\frak g}\in \hat{E}$.

For the Dedekind completeness of we observe that if $C\subset \hat{E}_+$ is a non-empty set bounded above by ${\frak g}\in \hat{E}_+$ then $C\subset {E^*}_+$ and is bounded above by ${\frak g}$ in $E^*$.
Thus ${\frak h}:=\sup C$ exists in $E^*$, as $E^*$ is Dedekind complete. Further ${\frak h}\le {\frak g}$ so there exists $k\in R(T)_+$ so that 
$$|{\frak h}(x)|\le |{\frak g}(x)|\le k\|x\|_{T,2}.$$
Thus ${\frak h}\in \hat{E}$. If  $\hat{{\frak h}}\in \hat{E}$ were an upper bound on $C$ in $\hat{E}$ then $\hat{{\frak h}}$ is also an upper bound for $C$ in $E^*$ making ${\frak h}\le \hat{{\frak h}}$, thus ${\frak h}$ is the least upper bound of $C$. 
\qed
\end{proof}

\begin{thm}
$T$ is a weak order unit for $\hat{E}$ .
\end{thm}

\begin{proof} Let ${\frak f}\in \hat{E}_+$.
For $x\in L^2(T)$ with $x\ge 0$ we have
\begin{eqnarray*}
\sup_{n\in\N} ({\frak f}\wedge nT)(x)&=&\sup_{n\in\N} \left(\inf_{u+v=x,u,v\ge 0}({\frak f}(u)+nT(v))\right)\\
&=&\sup_{n\in\N} \left(\inf_{u+v=x,u,v\ge 0}T(y({\frak f})u+nv)\right)\\
&=&\sup_{n\in\N} \left(\inf_{x\ge v\ge 0}T(y({\frak f})(x-v)+nv)\right)\\
&=&\sup_{n\in\N} \left(T(y({\frak f})x)+\inf_{x\ge v\ge 0}T(v(ne-y({\frak f})))\right)\\
&=&T(y({\frak f})x)+\sup_{n\in\N} \left(\inf_{x\ge v\ge 0}T(v(ne-y({\frak f})))\right).
\end{eqnarray*}
Here
$$-T(x(ne-y({\frak f}))^-)\le\inf_{x\ge v\ge 0}T(v(ne-y({\frak f})))\le 0$$
But $e$ is a weak order unit so $(ne-y({\frak f}))^-\downarrow 0$ in order as $n\to\infty$ giving that
$-T(x(ne-y({\frak f}))^-)\uparrow 0$ in order as $n\to\infty$. Thus
 $$\sup_{n\in\N} \left(\inf_{x\ge v\ge 0}T(v(ne-y({\frak f}))\right)=0$$
 and 
 $$\sup_{n\in\N} ({\frak f}\wedge nT)(x)=0$$
 making $T$ a weak order unit for $\hat{E}$.
\qed
\end{proof}

\begin{thm}\label{thm-order-pr}
$\Psi$ is a bijection between $E_+$ and $\hat{E}_+$. Hence $\Psi$ and $\Psi^{-1}$ are order preserving bijections. In particular $\Psi(f)^\pm=\Psi(f^\pm)$, $\Psi(f)\vee\Psi(g)=\Psi(f\vee g)$ and 
$\Psi(f)\wedge\Psi(g)=\Psi(f\wedge g)$, for all $f,g\in E$.
\end{thm}
\begin{proof}
From Theorem \ref{thm-final}, the map $\Psi:E\to\hat{E}$ is a bijection.
If ${\frak f}\in \hat{E}_+$  then there exists $f\in E$ such that 
 ${\frak f}=\Psi(f)$ and it follows from Lemma \ref{lem-RK} that
 \begin{eqnarray*}
0\le (\Psi(f)\wedge 0)(x)&:=&\inf\{T(fy)\,|\,y,z\in E_+, y+z=x\}\\
&=&
\inf\{T(fy)\,|\,0\le y\le x\}
\end{eqnarray*} 
for all $x\in E_+$. 
If $f\notin E_+$, then $f^-=0\vee (-f)\ne 0$. Let $P_{f^-}$ be the band projection onto the band generated by $f^-$ (all bands in $E$ are principal bands) and $p_{f^-}=P_{f^-}e\in E_+$. Taking $x=p_{f^-}$ gives
$$0\le {\frak f}(x)=({\frak f}\wedge 0)(x)=(\Psi(f)\wedge 0)(x)=(-Tf^-)\wedge 0\notin E_+,$$
since $T$ is a strictly positive operator which gives $Tf^->0$. Hence a contradiction and  $f\in E_+$.

If $f\in E_+$ then for $x\in E_+$, 
$$(\Psi(f)\wedge 0)(x)=\inf\{T(fy)\,|\,y,z\in E_+, y+z=x\}\ge 0$$
 as $y\ge 0$ and $T$ is a positive operator. Thus $\Psi(f)\in \hat{E}_+$ and $\Psi$ is a bijection between $E_+$ and $\hat{E}_+$.

The remaining claims of the theorem follow directly $\Psi$ be a linear bijection between $E_+$ and $\hat{E}_+$.
\qed
\end{proof}

\begin{thm}
$\Psi$ is a bijection between the components of $e$ in $E$ and the components of $T$ in $\hat{E}$.
\end{thm}

\begin{proof}
We begin by showing that $\Psi(p)$ is a component of $T$ for each $p$ a component of $e$. Using Theorem \ref{thm-order-pr}, we have $0\le \Psi(p)\le \Psi(e)=T$.
 Again using Theorem \ref{thm-order-pr}, we have
\begin{eqnarray*}
\Psi(p)\wedge \Psi(e-p)=\Psi(p\wedge (e-p))=\Psi(0)=0.
\end{eqnarray*} 
 Hence we have shown that $\Psi(p)$ is a component of $T$ for each $p$ a component of $e$.
  For the converse we merely work with $\Psi^{-1}$.
  \qed
\end{proof}

From the above, for components $p,q$ of $e$ in $E$. we have that
\begin{eqnarray}\label{comp-mult}
\Psi(p)\cdot \Psi(q)=\Psi(p\cdot q)
\end{eqnarray}
where multiplication is with respect to the $f$-algebra structures of $E_e$ and $\hat{E}_T$ with $e$ and $T$ being their respective algebraic units.

\begin{lem}\label{lem-oc}
$\Psi$ and $\Psi^{-1}$ are order continuous and 
\begin{eqnarray}\label{mult}
\Psi(p)\cdot \Psi(q)=\Psi(p\cdot q).
\end{eqnarray}
for all $p\in E_e$ and $q\in E$.
\end{lem}
\begin{proof}
Let $f_\alpha$ be a net in $E$ with order limit $f$ and let $g\in E$, then 
$$\Psi(f_\alpha)(g)=T(f_\alpha g)\to T(fg)$$
as pairwise multiplication is an order continuous map from  $L^2(T)\times L^2(T)\to L^1(T)$, and $T$ is order continuous on $L^1(T)$.  Thus giving that $\Psi$ is order continuous.

Suppose that $({\frak f}_\alpha)\subset \hat{E}_+$ is a downwards directed net with ${\frak f}_\alpha\downarrow 0$. Then  $(\Psi^{-1}({\frak f}_\alpha))\subset E_+$ is a downwards directed net with $\Psi^{-1}({\frak f}_\alpha)\downarrow h\ge 0$. So for $g\in E_+$ we have
$$T(\Psi^{-1}({\frak f}_\alpha)g)\downarrow T(hg)$$ from the order continuity of multiplication and of $T$.
However
$$T(\Psi^{-1}({\frak f}_\alpha)g)={\frak f}_\alpha(g)\downarrow 0.$$
Thus $T(hg)=0$ for all $g\in E_+$, in particular for $g=h$. Thus $T(h^2)=0$, which with the strict positivity of $T$ gives $h^2=0$, and thus $h=0$. Hence $\Psi^{-1}$ is order continuous.

As $\Psi$ is linear, (\ref{comp-mult}) extends immediately to $p$ and $q$ being linear combinations of components of $e$.  Applying Freudenthal's theorem along with the order continuity of $\Psi$ and multiplication, the result follows.
\qed
\end{proof}

It should be noted that (\ref{mult}) can be extended to $p\in L^\infty(T)$ where
$$L^\infty(T)=\{f\in E^u\,|\, |f|\le \alpha \mbox{ for some } \alpha\in R(T)\},$$
 see \cite{KRW} for more details on $L^\infty(T)$ and this multiplication.

We can characterize the band projections on $\hat{E}$.

\begin{thm}\label{thm-band}
The band projections on $\hat{E}$ are the maps $\hat{P}=\Psi P\Psi^{-1}$ where $P$ is a band projection on $E$.
\end{thm}

\begin{proof}
Let $\hat{P}$ be a band projection on $\hat{E}$ with range $\hat{B}$.  As all band in $\hat{E}$ are principal bands and as $T$ is a weak order unit for $\hat{E}$,  it follows that $\hat{P}T$ is a generator of $\hat{B}$, further it is a component of $T$.  Thus there is a component $p$ of $e$ in $E$ so that $\Psi(p)= \hat{P}T$.

Let ${\frak g}\in \hat{E}$ then the action of $\hat{P}$ on ${\frak g}$ is given by 
$$\hat{P}{\frak g}=(\hat{P}T)\cdot {\frak g}=\Psi(p)\cdot \Psi(\Psi^{-1}({\frak g})).$$
Applying (\ref{comp-mult}) we have
$$\Psi(p)\cdot \Psi(\Psi^{-1}({\frak g}))=\Psi(p\cdot \Psi^{-1}({\frak g}))=\Psi(P \Psi^{-1}({\frak g})),$$
where $P$ is the band projection in $E$ generated by $p$.
Thus $\hat{P}{\frak g}=\Psi(P \Psi^{-1}({\frak g})).$
\qed
\end{proof}

\begin{cor}
For each band projection $\hat{P}$ on $\hat{E}$ we have that the 
band projection $P:=\Psi^{-1}\hat{P}\Psi$ on $E$ has
$\hat{P}{\frak f}={\frak f} P, $ for all ${\frak f}\in\hat{E}$.
\end{cor}

\begin{proof}
Here $\hat{P}=\Psi P\Psi^{-1}$. Hence, for all ${\frak f}\in\hat{E}$ and $g\in E$,
$$(\hat{P}{\frak f})(g)=(\Psi P\Psi^{-1}{\frak f})(g)=\Psi(P\Psi^{-1}{\frak f})(g)=T(\Psi^{-1}{\frak f}\cdot Pg)={\frak f}(Pg)={\frak f}\circ P(g),$$
for all ${\frak f}\in \hat{E}$ and $g\in E$. Here we have used that 
$\Psi^{-1}{\frak f}\cdot Pg=P\Psi^{-1}{\frak f}\cdot g$.
\qed
\end{proof}

  \begin{note}
   Each ${\frak f}\in \hat{E}$ can be considered as a $R(T)$-valued Radon measure $\mu_{\frak f}$ defined on the components of $e$ in $E$ by $\mu_{\frak f}(p)={\frak f}(p)$. In this sense the Riesz representation of the measure is as $\mu_{\frak f}(p)=T(\Psi({\frak f})\cdot p)$ and $\Psi({\frak f})$ is the Radon-Nikod\'ym derivative of $\mu_{\frak f}$ with respect to $\mu_T$.
  \end{note}
\section{Completeness}

As in the scalar case, we can define a norm on $\hat{E}$, however in the case under consideration here this is an $R(T)$-valued norm or in the
notation of \cite{KRW} a $T$-norm. For ${\frak f}\in \hat{E}$ we define
\begin{equation}
 \|{\frak f}\|=\inf\{k\in R(T)^+\,|\,|{\frak f}(x)|\le k\|x\|_{T,2} \mbox{ for all } x\in E\}.\label{dual-norm}
\end{equation}
As $R(T)$ is Dedekind complete it follows that $\|{\frak f}\|\in R(T)^+$.

For each $x\in E$, $\{k\in R(T)^+\,|\,|{\frak f}(x)|\le k\|x\|_{T,2} \}$ is a closed (with respect to order limits in $E$) convex cone in $R(T)^+$. Thus
$$Y:=\bigcap_{x\in E} \{k\in R(T)^+\,|\,|{\frak f}(x)|\le k\|x\|_{T,2} \}$$
 is convex additive and  closed w.r.t. order limits suprema and infima in $R(T)^+$.
 Thus $Y$ has a unique minimal element 
$$\|{\frak f}\|:=\inf  \bigcap_{x\in E} \{k\in R(T)^+\,|\,|{\frak f}(x)|\le k\|x\|_{T,2} \}.$$
As $\|{\frak f}\|\in Y$ we have
$|{\frak f}(x)|\le \|{\frak f}\|\,\|x\|_{T,2}$ for each ${\frak f}\in \hat{E}$ and $x\in {L}^{2}(T)$.

\begin{thm}
 The dual space  $\hat{E}$ is strongly complete. 
\end{thm}

\begin{proof}
Let $({\frak f}_\alpha)$ be a strong Cauchy net in $\hat{E}$.
Since 
$$|{\frak f}_\alpha(x)-{\frak f}_\beta(x)|\le \|{\frak f}_\alpha-{\frak f}_\beta\| \,\|x\|_{T,2}$$
it follows that for each $x\in {L}^{2}(T)$, $({\frak f}_\alpha(x))$ is a Cauchy net in $R(T)$.
Eventually  
\begin{align*}
v_\alpha:=\sup_{\beta,\gamma\ge\alpha} \|{\frak f}_\beta-{\frak f}_\gamma\|_{T,2}
\end{align*}  
exists as an element of $R(T)$ and $v_\alpha\downarrow 0$.
Here by eventually we mean that there is there is $\beta$ in the index set of the net, so that for $\alpha \ge \beta$ the given statement holds.
So eventually 
\begin{eqnarray}\label{cauchy-bound}
|{\frak f}_\beta(x)-{\frak f}_\gamma(x)|\le v_\alpha\|x\|_{T,2},\quad\mbox{ for all } \beta,\gamma\ge \alpha,
\end{eqnarray}
 and the Cauchy net $({\frak f}_\alpha(x))$ is eventually bounded in $R(T)$.
 Here the net $(|{\frak f}_\beta(x)-{\frak f}_\gamma(x)|)$ is considered as being index by $(\alpha,\beta)$
 with $(\alpha,\beta)\le (\alpha_1,\beta_1)$ meaning $\alpha\le \alpha_1$ and $\beta\le \beta_1$.
We may thus define
$\underline{\frak f}(x):=\liminf_\alpha {\frak f}_\alpha(x)$, $\overline{\frak f}(x):=\limsup_\alpha {\frak f}_\alpha(x)$ in $R(T)$.
Here
$$0\le\overline{\frak f}(x)-\underline{\frak f}(x)
=\lim_\alpha(\sup_{\beta\ge\alpha}{\frak f}_\beta(x)-\inf_{\gamma\ge\alpha}{\frak f}_\gamma(x))
=\lim_\alpha\sup_{\beta,\gamma\ge\alpha}({\frak f}_\beta(x)-{\frak f}_\gamma(x))
\le \lim_\alpha v_\alpha\|x\|_{T,2}=0.$$
So we can set ${\frak f}(x):=\overline{\frak f}(x)=\underline{\frak f}(x)$ with $({\frak f}_\alpha(x))$ converging in order to ${\frak f}(x)$, see \cite{AKRW}, and hence
${\frak f}(x)$ is defined for each $x\in {L}^{2}(T)$.

From the linearity of order limits and the linearity of each ${\frak f}_\alpha$ it follows that ${\frak f}: {L}^{2}(T) \to R(T)$ is a linear map.
Taking the order limit in (\ref{cauchy-bound}) with respect to $\gamma$ gives
\begin{eqnarray}\label{limit-bound}
|{\frak f}_\beta(x)-{\frak f}(x)|\le v_\alpha\|x\|_{T,2},\quad\mbox{ for all } \beta\ge \alpha, x\in L^2(T).
\end{eqnarray}
Thus $$|{\frak f}(x)|\le |{\frak f}_\alpha(x)|+v_\alpha\|x\|_{T,2}\le (\|{\frak f}_\alpha\|+v_\alpha)\|x\|_{T,2}$$
for all $x\in L^2(T)$,
from which we have that ${\frak f}\in \hat{E}$.

Now  (\ref{limit-bound}) can be written as 
\begin{eqnarray*}
\|{\frak f}_\beta-{\frak f}\|\le v_\alpha,\quad\mbox{ for all } \beta\ge \alpha, 
\end{eqnarray*}
giving that
 $({\frak f}_\alpha)$ converges strongly to ${\frak f}$. 
\qed
\end{proof}

We recall from \cite{KRodW} some of the concepts of strong completeness as they relate to ${{L}}^2(T)$.

\begin{defs}
	We say that a net  $(f_\alpha)$ in ${{L}}^2(T)$, is a strong Cauchy net if 
	$$v_\alpha:=\sup_{\beta,\gamma\ge \alpha}\|f_\beta-f_\gamma\|_{T,2}$$
	is eventually defined and has order limit zero, i.e. there exists $\delta$ such that
        $v_\alpha$ is defined for all $\alpha\ge \delta$ and $\displaystyle{\inf_{\alpha\ge \delta} v_\alpha = 0}$.
\end{defs}
It should be noted that this is equivalent to requiring that $\|f_\beta-f_\gamma\|_{T,2}$ as a net indexed by $(\beta,\gamma)$ with componentwise directedness, converges to $0$ in order.
We are now in a position to give the definition of strong completeness.

\begin{defs}
        We say that ${{L}}^2(T)$ is strongly complete if each strong Cauchy net
 	$(f_\alpha)$ in ${{L}}^2(T)$, is strongly convergent in ${{L}}^2(T)$,
        i.e. there is $f\in {{L}}^2(T)$ so that  
	$$w_\alpha:=\sup_{\beta\ge \alpha}\|f_\beta-f\|_{T,2}$$
	is eventually defined and has order limit zero, i.e. there exists $\delta$ such that
        $w_\alpha$ is defined for all $\alpha\ge \delta$ and $\displaystyle{\inf_{\alpha\ge \delta} w_\alpha = 0}$.
\end{defs}

It should be noted that for the case of $E=L^2(\Omega,\Sigma,\mu)$ where $\mu$ is a finite measure and $Tf={\bf 1}\frac{1}{\mu(\Omega)}\int_\Omega f\,d\mu$,
where ${\bf 1}$ is the constant $1$ function, then ${{L}}^2(T)=E$, the vector norm $\|f\|_{T,2}=\|f\|_2$ is the standard $L^2$ norm and the 
concepts of strong Cauchy nets and strong completeness coincide with those of norm/strong Cauchy nets and norm/strong completeness.

\begin{thm}
 ${{L}}^2(T)$ is strongly complete.
\end{thm}

\begin{proof}
 The map $\Psi: y\mapsto T_y$ is an $R(T)$-linear $R(T)$-valued vector norm preserving bijection between ${{L}}^2(T)$ and $\hat{E}$.
 So given a strong Cauchy net $(y_\alpha)$ in ${{L}}^2(T)$, $(\Psi(y_\alpha)$ is a strong Cauchy net in $\hat{E}$ and is thus
 strongly convergent to some ${\frak f}\in \hat{E}$. As $\Psi$ is a bijection, there is $y \in {{L}}^2(T)$ so that ${\frak f}=\Psi(y)=T_y$.
 Due to the $R(T)$-linearity of $\Psi$ and this map be $R(T)$-valued vector norm preserving, the strong convergence of 
 $(\Psi(y_\alpha)$ to ${\frak f}$ in $\hat{E}$ yields the strong convergence of $(y_\alpha)$ to $y$ in ${{L}}^2(T)$.
\qed
\end{proof}

\end{document}